\documentclass[12pt]{article}

\usepackage{latexsym,amssymb,amsmath,bm}

\newcommand{\C}{\mathbb C}
\newcommand{\Z}{\mathbb Z}
\newcommand{\N}{\mathbb N}

\newcommand{\Q}{\mathbb Q}
\newcommand{\sma}{\left(\begin{array}}
\newcommand{\fma}{\end{array}\right)}

\newtheorem{lem}{Lemma}

\newtheorem{co}[lem]{Corollary}
\newtheorem{thm}[lem]{Theorem}
\newtheorem{prop}[lem]{Proposition}

\newenvironment{proof}{\textbf{Proof.}}{\newline\hspace*{\fill}{$\Box$}\\}

\begin{document}
\title{Tubular groups, 1-relator groups and non positive curvature}
\author{J.\,O.\,Button}

\newcommand{\Address}{{
  \bigskip
  \footnotesize

\textsc{Selwyn College, University of Cambridge,
Cambridge CB3 9DQ, UK}\par\nopagebreak
  \textit{E-mail address}: \texttt{j.o.button@dpmms.cam.ac.uk}
}}

\date{}
\maketitle
\begin{abstract}
We show (using results of Wise and of Woodhouse) that
a tubular group is always virtually special (meaning that it
has a finite index subgroup embedding in a RAAG)
if the underlying
graph is a tree. We also adapt Gardam and Woodhouse's
argument on tubular groups which double cover 1-relator groups to
show there exist 1-relator groups which are CAT(0) but not
residually finite.
\end{abstract}
\begin{center}
\end{center}
\hfill

\section{Introduction}

The usual notion of a discrete group $G$ having non positive curvature
is taken to be that it is CAT(0); namely it acts geometrically
(properly and cocompactly by isometries) on a CAT(0) metric space.
This certainly has some important consequences for the structure
of the group $G$ (see \cite{bh} Part III Chapter $\Gamma$
Theorem 1.1), for instance $G$ will be finitely presented.
However Wise's example in \cite{wsnh}
of a CAT(0) group which is not even Hopfian, thus not residually finite
nor linear over any field, shows that being CAT(0) is not a sufficient
condition to be well behaved in terms of abstract group theoretic
properties.

In recent years though, a stronger geometric notion has emerged,
namely that of acting geometrically not just on a CAT(0) space
but on a CAT(0) cube complex. Indeed because of the pioneering work
of Wise and Agol (\cite{hlws}, \cite{wsbig}, \cite{ag}),
if $G$ is also word hyperbolic then it is virtually special.
Here our definition of a virtually special group $G$ will be that
it has a finite index subgroup $H$, written $H\leq_f G$, where $H$ embeds
into a right angled
Artin group (a RAAG, which we always assume is finitely generated).
Being virtually special implies very strong group theoretic
consequences for $G$: indeed $G$ will inherit any property held
by all RAAGs and preserved under taking subgroups and finite index
supergroups. In particular $G$ will be linear not just over $\C$
but even over $\Z$ (in fact there are various groups only known to be linear
because they are virtually special, such as word hyperbolic free by
cyclic groups as shown in \cite{hwgen}), as well as
being virtually biorderable, virtually residually torsion free nilpotent
and virtually residually finite rational solvable for instance.

This interplay between the geometric and group theoretic properties
of $G$ is therefore very powerful if $G$ is word hyperbolic. However
in this paper our focus will be on groups which, although always
finitely presented (in fact all groups considered here will have a
presentation of deficiency 1), will usually contain $\Z\times\Z$.
The problem here is that in the absence of word hyperbolicity,
a finitely presented group
being virtually special neither implies nor
is implied by acting geometrically on a CAT(0) cube complex. To see this,
one way round follows from the finitely presented simple groups of
Burger and Mozes whereas there are finitely presented subgroups $H$ of the
RAAG $F_2\times F_2\times F_2$, so are virtually special, but
which do not have the correct finiteness properties even to be CAT(0) 
(\cite{bh} Chapter III.$\Gamma$ Section 5).

In order to investigate this further, we can restrict to classes of groups
which are non word hyperbolic but whose structure gives us a better chance
of determining which groups in the class behave well geometrically and
which ones have good group theoretic properties. The first class we will
consider here are the tubular groups: namely the fundamental group
${\cal G}(\Gamma)$ of a finite graph $\Gamma$
of groups with all vertex groups isomorphic to $\Z^2$ and all edge groups
isomorphic to $\Z$. Of course tubular groups can never be word hyperbolic
(and indeed are often not even relatively hyperbolic)
but they have already been shown in other papers to
possess a range of interesting behaviours.
For instance Wise's non Hopfian CAT(0) group
mentioned above is tubular but is not virtually special nor acts
geometrically on a CAT(0) cube complex. Indeed in \cite{wstr} Wise showed
that only a very few tubular groups act geometrically on a CAT(0) cube
complex (\cite{wstr} Corollary 5.10) and these few groups are all virtually 
special (\cite{wstr} Corollary 5.9). Furthermore a virtually special
tubular group will be CAT(0) (by \cite{wstr} Lemma 4.4) but certainly
not vice versa.

In particular if we want our tubular group $G$ to be well behaved from
a group theoretic point of view then showing it is CAT(0) will not
be enough, but showing it is virtually special certainly will be.
Thus in Section 2 we look for tubular groups which are virtually special
but which do not act geometrically on a CAT(0)
cube complex. 
The downside here is that confirming directly that a group $G$ is virtually
special can be very involved, because one usually needs Wise's machinery
to show not only that $G$ has a finite index subgroup $H$ which is the
fundamental group of a non positively curved cube complex but also that
the complex is special. However Woodhouse introduces in \cite{wdvs} a
technique specifically for tubular groups which is used to show
that a tubular group
is virtually special if and only if it acts freely on some locally finite
CAT(0) cube complex. This is achieved by adapting Wise's equitable sets
condition in \cite{wstr} (which is equivalent to the group having a
free action on some CAT(0) cube complex) and imposing extra constraints
on these equitable sets which imply that the group is virtually
special. Therefore it seems appropriate to see for which tubular groups
this result can be made to work. We show in Theorem \ref{vspe} that
Woodhouse's criterion is satisfied for any tubular group defined by
$({\cal G},\Gamma)$ where $\Gamma$ is a tree. Thus this gives a new
class of virtually special groups and we can conclude that they all possess
our strong group theoretic properties above. In particular these groups are
all linear over $\Z$ even though linearity of this family was
not previously known.

We finish in Section 3 by looking at 1-relator groups. Again we are
interested in the family of non word hyperbolic 1-relator groups
(it is open as to whether every word hyperbolic 1-relator group
is virtually special or acts geometrically on a CAT(0) cube complex
or even is CAT(0)). Here we only examine 2-generator 1-relator
groups, though we note the very recent work in \cite{lowil}
which, on combining Theorem 1.3, Conjecture 1.6, Corollary 1.7
and Conjecture 1.8,
would say that a 1-relator group with at least 3 generators 
is either word hyperbolic or it is hyperbolic relative to a
2-generator 1-relator group.

A well known source of 1-relator groups with bad behaviour, both
group theoretic and geometric, are the famous Baumslag - Solitar
groups $BS(m,n)=\langle a,b|ba^mb^{-1}=a^n\rangle$.
More precisely, those groups where $|m|=|n|$ have
$F_k\times\Z$ as a finite index subgroup and so are actually well behaved
in terms of non positive curvature (though they are obstructions to a
group being word hyperbolic). However those
Baumslag - Solitar groups where $|m|\neq |n|$ act as obstructions
even for non positively curved behaviour and we will refer to them
as unbalanced or non Euclidean Baumslag - Solitar groups. Indeed
a group containing any subgroup
isomorphic to an unbalanced Baumslag - Solitar group cannot 
be CAT(0), let alone act geometrically on a CAT(0) cube complex, nor
can it be linear over $\Z$ and it will certainly not be virtually
special. This then begs the natural question: given a 1-relator
group with no unbalanced Baumslag - Solitar subgroup, must it
be well behaved geometrically or group theoretically?

This was open until 
the construction in \cite{gw} of Gardam and
Woodhouse which obtained a very close relationship between certain
2-generator 1-relator groups and certain tubular groups. More
precisely, they consider the particular class of tubular groups
\[G_{p,q}=\langle a,b,s,t|[a,b]=1,sa^qs^{-1}=a^pb,ta^qt^{-1}=a^pb^{-1}
\rangle\]
for $p,q$ positive integers which (when $p>q$) are 
known as the snowflake groups introduced in \cite{brbr}.
In \cite{gw} it is shown that the group $G_{p,q}$
is an index 2 subgroup of a 2-generator 1-relator
group $R_{p,q}$. This allows them to transfer various group theoretic and
geometric properties (or the lack thereof)
from the tubular group to the 1-relator group. In particular 
they show that for $p\geq q$ the resulting group $R_{p,q}$ is not CAT(0).
However $R_{p,q}$ never contains 
an unbalanced Baumslag - Solitar subgroups, so this is the
first example of a 1-relator group which is not well behaved
geometrically even though it contains no unbalanced Baumslag - Solitar
subgroups. 

In Section 3 we look at these families $R_{p,q}$ and $G_{p,q}$
but vary the parameters so that $q>p$, whereupon \cite{gg} established
that $R_{p,q}$ is a CAT(0) group.
This allows us in Theorem \ref{hop} to give the first
examples of 1-relator groups which are CAT(0) but not residually finite.
We do this by showing that when $p=1$ and $q\geq 3$ is odd, the
resulting tubular subgroup $G_{p,q}$ is non Hopfian using a very
similar argument to Wise in \cite{wsnh}, which itself is an adaptation
of the original argument by Baumslag and Solitar. Consequently $G_{p,q}$
and $R_{p,q}$ are not residually finite. In fact in Theorem \ref{resf}
we adapt this argument, along with a result in \cite{kim},  to work out
exactly when the group $G_{p,q}$ (equivalently $R_{p,q}$) is residually
finite (here we can take $p,q$
to be any non zero integers): namely when $q$ divides $2p$. Thus on
regarding ``well behaved geometrically'' to mean CAT(0) and
``well behaved group theoretically'' to mean residually finite, we
have 1-relator groups $R_{p,p}$ which are well behaved group
theoretically but not geometrically, 1-relator groups $R_{p,3p}$
which are well behaved geometrically but not group theoretically,
and (for $p>3$) 1-relator groups $R_{p,p-1}$ which are neither,
all without unbalanced Baumslag - Solitar subgroups. This construction
also allows us to answer Question 20 in \cite{meor}, namely (for
$p$ even) the 2-generator 1-relator group $R_{p,p/2}$
is residually finite but has no finite index subgroup which is an
ascending HNN extension of a finitely generated free group.

\section{Virtually special tubular groups from trees}
If we are given a tubular group ${\cal G}(\Gamma)$ where $\Gamma$
is a tree then we build our group using repeated amalgamations
over $\Z$ and it is straightforward to see, for instance by applying
\cite{bh} Chapter II Corollary 11.19, that our tubular group will
be CAT(0). This is taken further in \cite{wstr} which establishes
other results for tubular groups.
The key definition there is that of an equitable set for $({\cal G},\Gamma)$.
This is defined to be a choice at each vertex $v$ in $\Gamma$
of a finite
number of closed curves $c^{(v)}_1,\ldots ,c^{(v)}_{m(v)}$ on a torus
$S^1\times S^1$ placed at $v$, which therefore can be regarded as
elements of $\pi_1(S^1\times S^1)=G_v\cong \Z^2$, so that on taking
any edge $e$ with endpoints $e_\pm$ and edge group $\langle z_e\rangle$
embedding in $G_{e_+}$ as $z_{e_+}\in\Z^2\setminus\{(0,0)\}$ and in $G_{e_-}$
as $z_{e_-}$, then the sum of the (algebraic, unsigned) intersection
numbers
\[\#[c^{(e_+)}_1,z_{e_+}]+\ldots +\#[c^{(e_+)}_{m(e_+)},z_{e_+}]\]
at $e_+$ is equal to the sum
\[\#[c^{(e_-)}_1,z_{e_-}]+\ldots +\#[c^{(e_-)}_{m(e_-)},z_{e_-}]\]
of intersection numbers at the other end. The importance of this
definition of Wise is his result in \cite{wstr} stating
that the tubular group defined by $({\cal G},\Gamma)$ acts
freely on a CAT(0) cube complex (which might well be infinite
dimensional and/or not locally finite) if and only if we can find an
equitable set for $({\cal G},\Gamma)$. 

Now we can regard $z_{e_+}$ as a
closed curve lying in the torus at $e_+$ and similarly $z_{e_-}$ at $e_-$,
with these two tori joined by a cylinder corresponding to the edge
$e$ in order to obtain a graph of spaces. Then this condition on
intersection numbers is needed to ensure a bijection between the
intersection points of $z_{e_+}$ with the equitable set at $v_+$ and
those of $z_{e_-}$ with the equitable set at $v_-$. (It is also required
as part of the definition of an equitable set that the elements at each
vertex of this set are not all parallel.)

Here we want to show that there are a range of tubular groups 
which are virtually special, in addition to those which act
properly and cocompactly on a CAT(0) cube complex. Now there
seem to be few criteria in existence which show a group is virtually
special without requiring this geometric action. However we have
a result of Woodhouse in \cite{wdvs} which does exactly this
for tubular groups, which we now review.

If a closed curve $c$ is not primitive, so that it is of the form
$n(a,b)\in\Z^2\setminus\{(0,0)\}$ with $n\geq 2$ and $a,b$ coprime,
then in the Wise result we can either regard $c$ as a primitive curve
traversed $n$ times, or as $n$ parallel disjoint primitive curves. In
the Woodhouse work the second approach is taken and we will do that here
too.

For Woodhouse's virtually special criterion, we first require an
equitable set consisting only of primitive elements (which can be
ensured by the comment above) but which is also fortified, meaning
that every time we inject an edge group $\langle z_e\rangle$ into a
vertex group $G_{e_+}$ (or $G_{e_-}$), the image $z_{e_+}$ of $z_e$ is
parallel to something in the equitable set at $e_+$, or alternatively
there is at least one element $c$ in this equitable set such that the
intersection number $\#[c,z_{e_+}]=0$.

A primitive, fortified equitable set is then used to create (primitive,
fortified immersed) walls. 
This can be thought of as using the equality between
intersection points on either side of an edge to create a graph $\Omega$
where each vertex of $\Omega$ corresponds to an element of the equitable
set (and so can be thought of as ``lying over'' a vertex $v$ of the
graph $\Gamma$). The edges of $\Omega$ are given by taking an edge $e$ of
$\Gamma$ and choosing a bijection $b$ between the intersection points
on either side of $e$, then adding an edge in $\Omega$ between the two
elements of the equitable set in which the pair of intersection points
$p_\pm$ lie, for each pair where $b$ satisfies $b(p_-)=p_+$.    
The resulting graph $\Omega$, which certainly need not be
connected, is the space of immersed walls (actually this space is where
each vertex is not just a point, rather a copy of $S^1$ but for the result 
that follows in \cite{wdvs} each of these copies of $S^1$ can be replaced 
with a single point and we will use this construction). Note that our
graph $\Omega$ can be thought of as having the graph $\Gamma$ as a
``quotient'', in that we have a projection respecting edges which
sends an element of our equitable set to the vertex of $\Gamma$ in
which it lies, and the same with the edges of $\Omega$ and of $\Gamma$.

Let the connected components of $\Omega$ be $\Omega_1,\ldots ,\Omega_k$.
The criterion that we will now use is \cite{wdvs} Proposition 4.8, which
states that if $\Omega_1,\ldots ,\Omega_k$ are primitive, fortified,
undilated immersed walls then the tubular group $G=\pi_1({\cal G})$
is virtually special, with only the condition of being undilated left
to explain. We describe this as follows: suppose we have a directed
edge path $e_1,\ldots ,e_n$ in the graph $\Omega$. The dilation
function from such edge paths to $\Q^*$ is calculated as follows:
we start with value 1 at the vertex $(e_1)_-$, which is an element of
our equitable set, and then as we traverse the edge $e_1$ to arrive at
the vertex $(e_1)_+$, we multiply our value by the ratio of intersection
numbers 
$\#[(e_1)_-,z_{(\overline{e_1})_-}]/\#[(e_1)_+,z_{(\overline{e_1})_+}]$.
Here $\overline{e_1}$ is the edge in $\Gamma$ lying below the edge $e_1\in 
\Omega$ and
$z_{(\overline{e_1})_\pm }$ are the respective inclusions 
at each end of the
generator $z_{(\overline{e_1})}$ of the edge group for
$\overline{e_1}$ in $\Gamma$. We then multiply our value by
$\#[(e_2)_-,z_{(\overline{e_2})_-}]/\#[(e_2)_+,z_{(\overline{e_2})_+}]$  
as we cross the edge $e_2$ in $\Omega$ and so on, thus obtaining the value of
our dilation function for this  edge path when we arrive at the
final vertex $(e_n)_+$. We then say that $\Omega$ has undilated immersed
walls if the dilation function of every closed path in $\Omega$ is 1. 

Now although our original graph $\Gamma$ is a tree, the components
of $\Omega$ need not be. Even if they are, this does not ensure that
$\Omega$ has undilated immersed walls, as demonstrated in
\cite{wdph} Example 4.6.2 where the tubular group
$\langle a,b,c,d|[a,b],[b,c],[c,d]\rangle$ is taken, along with an equitable
set and pairing of intersection points to create a graph $\Omega$ which
is in fact connected but where there exists a dilated immersed wall.
As this group is clearly a RAAG, we see that we have to take care both in
providing an appropriate equitable set and in choosing suitable
bijections between intersection points in order to conclude a group is
virtually special using this criterion.

The following is the crucial point that we will use in order to obtain
undilated immersed walls.
\begin{prop} \label{dil}
Let $G$ be a tubular group obtained from the graph of groups
$({\cal G},\Gamma)$ where $\Gamma$ is a tree. Suppose that we have a
fortified, primitive equitable set for $({\cal G},\Gamma)$ and for
each edge $e$ in $\Gamma$ a choice of bijection between the intersection
points of the inclusion of the edge group generator (written 
$z_{e_\pm}$ at each vertex $e_\pm$) with the equitable set at these vertices
$e_\pm$. This gives rise to our graph $\Omega$.
If $\Omega$ satisfies the following property:\\
\hfill\\
Whenever we traverse two edges $e_1,e_2$ in the same connected component
of $\Omega$ (from $(e_1)_-$ to
$(e_1)_+$ and $(e_2)_-$ to $(e_2)_+$) such that both of their
projections in the graph $\Gamma$ are the same edge $e$ traversed in the
same direction then the intersection number
$\#[(e_1)_-,z_{e_-}]$  is equal to $\#[(e_2)_-,z_{e_-}]$  and 
$\#[(e_1)_+,z_{e_+}]$  is equal to $\#[(e_2)_+,z_{e_+}]$\\
\hfill\\
then $\Omega$ has undilated immersed walls and so $G$ is a virtually
special group by \cite{wdvs} Proposition 4.8.
\end{prop}
\begin{proof}
Suppose that $e_1\ldots e_n$ is a closed edge path in the graph $\Omega$,
so that the vertices $(e_1)_-$ and $(e_n)_+$ are the same. This path lies
above an edge path in the original graph $\Gamma$, which we will write
$\overline{e_1}\ldots \overline{e_n}$ and which is also closed. As
$\Gamma$ is a tree, we must have that two successive edges,
$\overline{e_i}$ and $\overline{e_{i+1}}$ say, are the same edge
$\overline{e}$ in $\Gamma$ but traversed in opposite directions.
As for the dilation of $e_1\ldots e_n$, this is equal to 
\begin{equation} \label{dlpr}
\frac
{\#[(e_1)_-,z_{(\overline{e_1})_-}]}
{\#[(e_1)_+,z_{(\overline{e_1})_+}]}\ldots
\frac
{\#[(e_i)_-,z_{(\overline{e_i})_-}]}
{\#[(e_i)_+,z_{(\overline{e_i})_+}]}
\frac
{\#[(e_{i+1})_-,z_{(\overline{e_{i+1}})_-}]}
{\#[(e_{i+1})_+,z_{(\overline{e_{i+1}})_+}]}
\ldots
\frac
{\#[(e_n)_-,z_{(\overline{e_n})_-}]}
{\#[(e_n)_+,z_{(\overline{e_n})_+}]}.
\end{equation}
Now our edge path $e_1\ldots e_n$ is of course always in the same
connected component of $\Omega$. As for our pair of edges $e_i,e_{i+1}$
projecting down to the same edge $\overline{e}\in\Gamma$ but in
opposite directions, our condition in the hypothesis above tells us that
$\#[(e_i)_+,z_{(\overline{e_i})_+}]=
\#[(e_{i+1})_-,z_{(\overline{e_{i+1}})_-}]$ and also
$\#[(e_i)_-,z_{(\overline{e_i})_-}]=
\#[(e_{i+1})_+,z_{(\overline{e_{i+1}})_+}]$.
Thus we can remove these edges from our path $e_1\ldots e_n$ and also
the corresponding intersection numbers from our product in (\ref{dlpr})
without changing the overall dilation of this edge path. Although we may
no longer have an edge path in $\Omega$ on taking away the edges $e_i,
e_{i+1}$, its projection to the tree $\Gamma$ is still a closed
edge path. Consequently we can now find successive edges $e_j,e_{j+1}$
where $\overline{e_j},\overline{e_{j+1}}$ are the same edge in $\Gamma$
with opposite orientations, thus this pair of edges can also be
removed from $\Omega$ and their projections from $\Gamma$. Again our
condition on intersection numbers means that the dilation of our original
path in $\Gamma$ is unchanged on removal of the relevant factors. Now
we can keep removing pairs of edges that ``cancel'' in turn until we
are left with the empty product in (\ref{dlpr}), thus our immersed
walls are undilated.
\end{proof}
\begin{thm} \label{vspe}
Let $G$ be a tubular group obtained from the graph of groups $({\cal G},
\Gamma)$ where $\Gamma$ is a tree. Then there exists a fortified, primitive,
equitable set for $({\cal G},\Gamma)$ such that the graph $\Omega$ obtained
as above satisfies the property in Proposition \ref{dil}. Consequently
$G$ has a finite index subgroup that embeds in a RAAG.
\end{thm}
\begin{proof}
The construction of our equitable set can be summarised
as follows: at each vertex in turn,
examine the images of all edge groups embedding in this vertex group and
begin by taking the primitive element in each parallelism class of these
images. These must be included if we are to obtain a fortified and
primitive equitable set for $({\cal G},\Gamma)$, though we will also add
some other arbitrary primitive elements to some vertices until these sets
at each vertex are the same size. At this point we have not considered
intersection numbers at all and so have made no attempt to create an
equitable set for $({\cal G},\Gamma)$, let alone one with undilated walls.
However both of these conditions will be achieved by taking many parallel
copies of the elements at each vertex. The number of copies will be chosen
when we add a new vertex of $\Gamma$ at each stage, by using the
intersection number with the new edge group to balance out the copies of
these sets at both ends of our edge so that all intersection numbers are
equal. However we will need to backtrack and take  more copies of our
sets at the vertices which have already been added.

In more detail: let $v_0$ be the vertex of $\Gamma$ with the most parallelism
classes of edge groups (in fact \cite{wstr} Corollaries 5.9 and 5.10
tell us that if this number is at most 2 then $G$ is already virtually
special, so we can assume that we have at least 3 to avoid small cases).
Say there are $m$ ($\geq 3$) classes at $v_0$ and take
${\bf x}_1^{(0)},\ldots ,{\bf x}_m^{(0)}$ to be the primitive elements
in each parallelism class. Now label the other vertices consecutively
using a breadth based order with $v_0$ at the root. Then at each other 
vertex $v_i$ take
a primitive element for each parallelism class of edge groups in the vertex
group $G_{v_i}$, which we will label as
${\bf x}_1^{(i)}, \ldots ,{\bf x}_{m_i}^{(i)}$ and where we have
$m_i\leq m$. Next we add arbitrary new primitive elements at each vertex
until we have $m$ of them, thus obtaining
${\bf x}_1^{(i)},\ldots ,{\bf x}_m^{(i)}$ at each vertex $v_i$ where now
$m$ is independent of $i$ and these $m$ primitive elements are all
distinct.

We now proceed through the tree, vertex by vertex, to build our equitable
set out of repeated copies of these elements. We will end up by building
$m$ subgraphs $\Omega_i$, which need not be connected, but such that
$\Omega$ is the disjoint union of $\Omega_1,\ldots ,\Omega_m$. First take
the base $v_0$ and an adjacent vertex $v_1$, connected by the edge $e$.
On taking as usual $z_{e}$ to be a generator of the edge group, with
inclusions $z_{e_-}$ in the vertex $v_0$ and $z_{e_+}$ at $v_1$, our
first priority is to examine the intersection numbers
\begin{eqnarray*}
i_1=\#[{\bf x}_1^{(0)},z_{e_-}],&\ldots &,i_m=
\#[{\bf x}_m^{(0)},z_{e_-}]\qquad\mbox{ and}\\
j_1=\#[{\bf x}_1^{(1)},z_{e_+}],&\ldots &,j_m=
\#[{\bf x}_m^{(1)},z_{e_+}].
\end{eqnarray*}
In each list of numbers, exactly one is zero and the rest are positive
because at each vertex the primitive element parallel to this edge group
was included, but no other element parallel to it was included. Thus all
other numbers in each list are strictly positive, so here we renumber to
make $i_1=j_1=0$. We now start to create the graph $\Omega_1$ by placing
a vertex $\omega_1^{(0)}$ above $v_0$, a vertex $\omega_1^{(1)}$ above
$v_1$ and no edge between them. Moving now to $i_2$ and $j_2$, there is
no a priori relationship between them, so no reason to assume $i_2=j_2$
which would enable the construction of an equitable set. But we are free
to take repeated copies of these elements on both sides, so we begin the
construction of the subgraph $\Omega_2$ by placing $j_2$ separate vertices
above $v_0\in\Gamma$ and $i_2$ vertices above $v_1\in\Gamma$.
The intersection numbers tell us that each vertex in $\Omega_2$ above
$v_0$ has $i_2$ edges leaving it and all of these edges need to end up
at a vertex above $v_1$, But similarly we have $j_2$ edges each arriving
at the $i_2$ vertices above $v_1$, so we can pair up all of these $i_2j_2$
edges bijectively so that they run between our vertices above $v_0$ and
$v_1$ (for instance in a $j_2$ by $i_2$ ``rectangular grid''). We then
proceed in this way through the 3rd,$\ldots$, $m$th pair of intersection
numbers until we have (partial) subgraphs $\Omega_1,\ldots , \Omega_m$
all lying above $v_0,v_1$ and $e$ in $\Gamma$.

We now take our next vertex $v_2$, which will also joined to the base vertex 
$v_0$ by the edge $f$ say. At $v_0$ our set is currently a single copy of
${\bf x}_1^{(0)}$, then $j_2$ copies of ${\bf x}_2^{(0)}$,$\ldots$,
$j_m$ copies of ${\bf x}_m^{(0)}$ whereas for our new vertex $v_2$ it is
just single copies each of ${\bf x}_1^{(2)},\ldots ,{\bf x}_m^{(2)}$.
Let us now examine the corresponding intersection numbers; these
give us new integers
\begin{eqnarray*}
k_1=\#[{\bf x}_1^{(0)},z_{f_-}],&\ldots &,k_m=\#[{\bf x}_m^{(0)},z_{f_-}]
\mbox{ and}\\
l_1=\#[{\bf x}_1^{(2)},z_{f_+}],&\ldots &,l_m=\#[{\bf x}_m^{(2)},z_{f_+}]
\end{eqnarray*}
respectively. Once again, there is a single zero in each list and we permute
the $l$s so that $k_r=l_r=0$. Now we match up the sets across the edge
$f$: assuming that $r>1$ (as if $r=1$ then we can proceed exactly as
before), we change our single copy of 
${\bf x}_1^{(0)}$ at $v_0$ to $l_1$ copies of ${\bf x}_1^{(0)}$, then put
$k_1$ copies of ${\bf x}_1^{(2)}$ at $v_2$. In terms of the subgraph
$\Omega_1$, what was previously just a single vertex above $v_1$ 
remains as only one above $v_1$, but now we have $l_1$  vertices
above $v_0$, each with $k_1$ edges leaving in the direction of $f$, and we
also have $k_1$ vertices above $v_2$, each with $l_1$ edges arriving from
$f$. Once again we pair off these $k_1l_1$ edges on either side with each
other in some bijective correspondence.

We now update the subgraph $\Omega_2$, which currently has $j_2$ vertices
above $v_0$ corresponding to copies of the element ${\bf x}_2^{(0)}$ and
$i_2$ vertices above $v_1$ corresponding to elements ${\bf x}_2^{(1)}$.
In order to match the intersection numbers $k_2$ at $v_0$ and $l_2$ at
$v_2$, we increase the $j_2$ copies of ${\bf x}_2^{(0)}$ at $v_0$ to
$j_2l_2$ copies and we now put $j_2k_2$ copies of ${\bf x}_2^{(2)}$ at
$v_2$. This means that in $\Omega_2$ we have two lots of $j_2k_2l_2$ half
edges lying above the edge $f$ which  can be paired off and joined up.

This is fine for the new vertex $v_2$ but we have changed the number of
copies of ${\bf x}_2^{(0)}$ at $v_0$ from $j_2$ to $j_2l_2$. However these
$j_2$ vertices had a total of $i_2j_2$ edges joined to the $i_2$ vertices
at $v_1$. Consequently we now need to change the equitable set at $v_1$
from $i_2$ copies of this element ${\bf x}_2^{(1)}$
to $i_2l_2$ copies, so that the extra
edges now above $v_0$ in the direction of $e$ can all be joined up to these
vertices above $v_1$.

We continue in this way where we add to $\Omega_s$ (for
$s=3,\ldots ,m$ in turn)
the appropriate number of vertices above $v_2$ and the correct
number of multiples of our vertices in $\Omega_s$ that already lie above 
$v_0$, so that we can create a bijection between these intersection points,
and then replicate the same number of multiples of the vertices in $\Omega_s$
above $v_1$. (For $s=r$ though, we have intersection number zero
at $v_2$ which means there are no intersection points that need to
be matched up. Here we merely place a single copy of ${\bf x}_r^{(2)}$
in $\Omega_r$ above $v_2$ and add no edges.) 

Now we proceed to the other vertices in order of our labelling, thus
first working through the edges attached to $v_0$ and then to its
descendents and so on. On adding a new edge $e$ and new vertex $v_n$
joined by $e$ to a previous vertex $v_p$,
at the $s$th stage we take our current copy of $\Omega_s$ and place
the correct number of copies of ${\bf x}_s^{(n)}$ above $v_n$,
as determined by the intersection numbers of ${\bf x}_s^{(p)}$ and
${\bf x}_s^{(n)}$ with $e$ (or one copy if that latter number is zero,
then skipping the next step),
along the appropriate number of multiples of the copies of ${\bf x}_s^{(p)}$
already in place above $v_p$. We then replicate the same number
of multiples of every other
vertex and edge in $\Omega_s$ that has been placed so far
(in fact, it is enough
just to do this within the current connected component of $\Omega_s$).
We finish with the leaf vertices, thus obtaining the final graphs
$\Omega_1,\ldots ,\Omega_m$.

Thus we finish with an equitable set for $({\cal G},\Gamma)$ because
at each stage we made sure to create the necessary bijections
between individual intersection points, thus also between the unions
of these intersection points on either side of any edge. The set is
also clearly fortified and primitive. As for the condition in Proposition
\ref{dil}, our method of construction for $\Omega$ was that whenever
we were in a single subgraph $\Omega_s$, thus certainly within
a connected component of $\Omega$, all edges in $\Omega_s$ that were
placed over the same edge in $\Gamma$ were joined at each end to
possibly different copies of the same element ${\bf x}_s^{(?)}$ say
at one end and ${\bf x}_s^{(??)}$ at the other, thus
at each end 
the intersection numbers with the edge subgroup generator
are all equal. Thus Proposition \ref{dil} applies and we conclude
that $\pi_1({\cal G})$ is a virtually special group. 
\end{proof}

As an immediate Corollary we obtain linearity for all of these
groups, because RAAGs are linear (even over $\Z$).
\begin{co} \label{cof}
If the tubular group $G$ is defined by the graph of groups
$({\cal G},\Gamma)$ where $\Gamma$ is a tree then $G$ is linear
over $\Z$.
\end{co}
In this regard, we note the conjecture of Metaftsis, Raptis and
Varsos in \cite{mrv} that the fundamental group of a graph of groups,
where the graph is a finite tree and the
vertex groups are all finitely generated abelian, is linear. Thus this result
establishes their conjecture in the case of $\Z^2$ vertex groups and
$\Z$ edge groups. But clearly not all tubular groups are linear.

We also note here that these groups $G$ in Corollary \ref{cof} are all
free by cyclic, by \cite{mecan} Corollary 2.2, thus obtaining more
free by cyclic linear groups. It is unknown whether all free by
cyclic groups are linear (this is open for Example 2 below).
Moreover these groups $G$ all have linearly growing monodromy,
whereas in the recent paper \cite{briv} it was shown by direct
construction of special cube complexes that for each $k\in \N$
there is a virtually special free by cyclic group with monodromy
growth polynomial of degree $k$.

We finish this section with a few basic examples to illustrate the
different types of geometric and group theoretic behaviour that
we see in tubular groups. We first note that \cite{wstr} Corollary 5.10
gives a complete classification of the tubular groups acting
geometrically on a CAT(0) cube complex: they have exactly one or two
parallelism classes of edge groups in each vertex group and do not contain
any unbalanced Baumslag - Solitar subgroup. In particular we see from Theorem
\ref{vspe} that there are many tubular groups which are
virtually special but which do not act geometrically on any 
CAT(0) cube complex: namely taking $({\cal G},\Gamma)$ with $\Gamma$ a tree
that contains a star with three edges and such that each inclusion into the
central vertex group of this star is in a different
parallelism class. (As $\Gamma$ is a tree, we already know that the
tubular group is CAT(0) and so contains no unbalanced Baumslag - Solitar
subgroups.) As for other types of behaviour, we now present two famous
examples. In both cases $\Gamma$ is a single point with vertex group 
$\langle a,b\rangle\cong\Z^2$ joined by two self loops.\\
\hfill\\
{\bf Example 1}: Wise's tubular group $G$, which was shown to be
both non Hopfian and CAT(0) in \cite{wsnh}, is given by the
two pairs of inclusions $a$ and $a^2b^2$, $b$ and $a^2b^2$.
Thus it is a tubular group which is CAT(0) but not virtually
special. (If it were then the
resulting finite index special group would be residually finite as a 
subgroup of a RAAG, hence so would $G$. However it is well known that
for finitely generated groups, being residually finite implies being
Hopfian.) We will see some very similar groups in the next section.\\
\hfill\\
{\bf Example 2}: Gersten's free by cyclic group, which was
shown not to be a CAT(0) group in \cite{ger},
can be expressed as a tubular
group where one edge provides inclusions $ab$ and $b$, and the other
$a^{-1}b$ and $b$. Thus it is an example of a tubular group which is
not CAT(0) but which is residually finite (as a free by cyclic group)
and hence Hopfian, and which does not contain an unbalanced
Baumslag - Solitar subgroup (for instance by \cite{mel} Proposition 2.5).

\section[A 1-relator group from a tubular group]{A 1-relator group 
from a tubular group, following Gardam
and Woodhouse}

In \cite{gw} the following family of presentations is
considered:
\[\langle x,y,t|x^2=y^2, t^{-1}x^{2q}t=x^{2p-1}y\rangle  \]
where $p,q$ are positive integers. The group $R_{p,q}$ so defined
has of course a 2-generator 1-relator presentation (just eliminate $y$).
Interestingly it is shown that $R_{p,q}$ has an index 2 subgroup
$G_{p,q}$ with
the following presentation:
\[G_{p,q}=\langle a,b,s,t|[a,b],s^{-1}a^qs=a^pb,t^{-1}a^qt=a^pb^{-1}\rangle
\]
which we immediately recognise as a tubular group in line with the
examples in Section 4. Moreover on taking $p>q\geq 1$, these are 
exactly the snowflake groups in \cite{brbr} which have unusual Dehn functions
that are greater than quadratic and so they cannot be CAT(0). Thus
\cite{gw} provides the first examples of 1-relator groups which are
not CAT(0) 
but which do not contain unbalanced Baumslag - Solitar subgroups.
Also \cite{gw} Proposition 4 shows that $R_{p,q}$ does not act freely
on a CAT(0) complex for $p>q$, because $G_{p,q}$ does not (by applying
Wise's equitable sets condition). 

But what about the 1-relator groups $R_{p,q}$ where $q\geq p\geq 1$? 
In this case
we still have the index 2 tubular subgroup $G_{p,q}$ with presentation
as above. For $p=q=1$ \cite{gw} points out that $G_{1,1}$ is none
other than Gersten's group. Thus $R_{1,1}$ is a 1-relator group that
is virtually free by cyclic (indeed it is free by cyclic by K.\,S.\,Brown's
criterion) with quadratic Dehn function and which acts freely on a CAT(0)
complex but which is not CAT(0).

Gardam showed in his thesis \cite{gg}
that the group $R_{p,q}$ for $p,q\geq 1$ is CAT(0) exactly when $q>p$.
This suggests trying to mimic Wise's non Hopfian CAT(0) group
construction from \cite{wsnh} which was Example 1 in the previous section.
\begin{thm} \label{hop}
For odd $q>p=1$, the tubular group $G_{1,q}$ above
is not Hopfian.
\end{thm}
\begin{proof}
We have
\[G_{1,q}:=\langle a,b,s,t|[a,b],s^{-1}a^qs=ab,t^{-1}a^qt=ab^{-1}\rangle,\]
so can take $\theta:G_{1,q}\rightarrow G_{1,q}$ which sends $a$ to $a^q$
and $b$ to $b^q$ but fixes $s$ and $t$. This preserves all three relations
of $G_{1,q}$ so is a well defined homomorphism. Now $s,t,a^q,b^q$ are
all in the image of $\theta$, thus so too is $ab$ and $ab^{-1}$. As $q$
is odd, we also get $a$ and $b$ so $\theta$ is onto.

To show $\theta$ is not injective, we apply the usual properties of HNN 
extensions as in \cite{wsnh} and even in the original construction of
Baumslag - Solitar groups. The commutator
\[ [s^{-1}as,ab^{-1}]=s^{-1}asab^{-1}s^{-1}a^{-1}sa^{-1}b\]
is a non trivial element of $G_{1,q}$ by Britton's Lemma for multiple
HNN extensions as there are no pinches, but it maps to the
identity as $\theta(s^{-1}as)=ab$.
\end{proof}
\begin{co}
The 1-relator group $R_{1,q}$ for odd $q\geq 3$ is CAT(0)   but
contains the index 2 subgroup $G_{1,q}$ which is not residually
finite, thus nor is $R_{1,q}$.
\end{co}
In particular $R_{1,q}$ is not virtually special. We believe this is the
first known 1-relator group which is CAT(0) but not residually
finite.

We now consider the question of when the group $G_{p,q}$ (equivalently
$R_{p,q}$) is residually finite. Here we will not insist that $p,q$
are positive, rather we will allow them to be any non zero integers.
By Malce'ev's result that a finitely
generated residually finite group is Hopfian, we know that any element
in the kernel of a surjective endomorphism must lie in the intersection
of all finite index subgroups, motivating the next proof.
\begin{thm} \label{resf}
The group $G_{p,q}$ is not residually finite if $q$ does not divide
$2p$. 
\end{thm}
\begin{proof}
Let $h=(2p,q)$ be the highest common factor of $2p$ and $q$, so that $0<h<q$.
Consider the element $x=[s^{-1}a^hs,a^pb^{-1}]$ which, as in Theorem \ref{hop}
above, has no pinches and so is a non trivial element of $G_{p,q}$.

Now take any homomorphism from $G_{p,q}$ to a finite group $F$
and let $r$ be the order of (the image of)
$a$ in $F$, so that $a^q$ has
order $d$ where $d=r/(r,q)$. This means that both $a^pb$ and $a^pb^{-1}$
have order $d$ too, but $a$ and $b$ still commute in $F$ so that
$e=a^{2pd}=b^{2d}$. Hence $r$ divides $2pd$ and so $(r,q)$ divides $2p$.

Consequently $(r,q)$ divides both $2p$ and $q$, hence $h$ also.
Thus there will exist integers $u,v$ with $ru+vq=h$, so that in $F$
we have $s^{-1}a^hs=s^{-1}a ^{ru+vq}s=(s^{-1}a^qs)^v$. Hence $s^{-1}a^hs$
is a power of $s^{-1}a^qs=a^pb$ in $F$ and will therefore
commute with any elements
that commute with $a^pb$, in particular $a,b$ and $a^pb^{-1}$.
Thus the element $x\in G_{p,q}$ above is non trivial, but trivial in any
finite quotient.
\end{proof}

Theorem 3.7 of \cite{kim} shows that a tubular group (or indeed the
fundamental group of any graph of finitely generated free abelian groups
with infinite cyclic edge groups) is residually finite
if each edge inclusion sends the generator of the edge group $\Z$ to
 a primitive element of the corresponding vertex group. This allows
 us to obtain a complete characterisation of the residually
finite groups in the $G_{p,q}$ family, and hence in the $R_{p,q}$
family too.

\begin{co} The group $G_{p,q}$ is residually finite if and only if
$q$ divides $2p$.
\end{co}
\begin{proof} The elements $a^pb$ and $a^pb^{-1}$ are certainly
primitive in the edge group $\langle a,b\rangle$, so if $q=\pm 1$ then
the above result applies. Otherwise first suppose that $q$ divides $p$. Then
we follow the idea in \cite{metb} Proposition 3.1. We have a homomorphism
from $G_{p,q}$ to the cyclic group $C_q$ given by sending $a$ to 1
and $b,s,t$ to zero. Now $G_{p,q}$ is a tubular group and thus acts on its
Bass - Serre tree $T$, so the kernel $K$ has finite index $q$
in $G_{p,q}$ and also acts on $T$ with $\Z^2$ vertex stabilisers and
$\Z$ edge stabilisers. This gives rise to a decomposition of $K$ as a
tubular group, with edge groups $E\cap K$ and vertex groups $V\cap K$
where $E$ and $V$ are the edge and vertex groups for $G$ on $T$, thus
are conjugates in $G$ of $\langle a,b\rangle$ and $a^q,a^pb,a^pb^{-1}$
respectively. Now $\langle a,b\rangle\cap K=\langle a^q,b\rangle$ and
as $K$ is normal, this will hold for any conjugate of $\langle a,b\rangle$
too. Thus now the edge inclusions of $a^q,a^pb,a^pb^{-1}$ are all primitive 
elements in the vertex groups of $K$.

If $q$ does not divide $p$, so that $q=2m$ is even with $m$ dividing
$p$, we proceed in exactly the same way but now we use the homomorphism
from $G_{p,q}$ to $C_q$ sending $a$ to 1, $b$ to $p$ and $s,t$ to 0.
\end{proof}

We note that the forthcoming paper \cite{hwdws} characterises exactly
when a tubular group is residually finite.

Finally we can also answer Question 20 in \cite{meor}
from 2010, which was put forward to suggest that we knew very little
about how widespread residually finite 2-generator 1-relator groups might
be. It says:\\
\hfill\\
Let $G$ be a group with a 2-generator 1-relator presentation
where the relator is not a proper power. Suppose that $G$ is residually
finite then does $G$
have a finite index subgroup $H$ which is an ascending HNN extension
of a finitely generated free group?\\
\hfill\\
Using the above constructions, we have:
\begin{co} The group $R_{p,q}$ above where $p>q\geq 1$ and $q$ divides
$2p$ is residually finite, but no finite index subgroup of $R_{p,q}$
is an ascending HNN extension of a finitely generated free group.
\end{co}
\begin{proof} For these values of $p,q$ we know from above that
$R_{p,q}$ is residually finite but has a Dehn function which is
bigger than quadratic. By \cite{brgr} all free by cyclic groups
$F_n\rtimes_\alpha\Z$ have at most quadratic Dehn function and
this is a quasi isometry invariant, so $R_{p,q}$ has no finite
index subgroup of this form either.

However this does not rule out $R_{p,q}$ having a finite index subgroup
which is is a strictly ascending HNN extension of a finitely generated
free group (as such groups can contain Baumslag - Solitar subgroups
of the form $BS(1,q)$ and so have exponential Dehn function). For this,
we note that the corresponding tubular group $G_{p,q}$ will have
such a finite index subgroup $H$ too, by \cite{melrg} Proposition 4.3
(iii). Now $H$ will also be a tubular group, but \cite{cshlev} Corollary 2.10
states that a tubular group has a symmetric BNS invariant and so
cannot be a strictly ascending HNN extension of any finitely generated
group.
\end{proof}

\Address

\end{document}